\documentclass[reqno,12pt,letterpaper]{amsart}
\usepackage{amssymb,oldlfont}
\usepackage{amsmath}
\usepackage{mathrsfs}
\usepackage{hyperref}

\usepackage[draft]{changes}
\usepackage{color}
\def\wrtext#1{\relax\ifmmode{\leavevmode\hbox{#1}}\else{#1}\fi}

\def\begeq{\begin{equation}}
	\def\endeq{\end{equation}}

\let\epsilon=\varepsilon

\def\part#1{\frac{\partial}{\partial #1}}

\newcommand{\Op}{\wrtext{Op}}

\newcommand{\DD}{{\mathcal D}}
\newcommand{\HH}{{\mathcal H}}
\newcommand{\LL}{{\mathcal L}}
\newcommand{\OO}{{\mathcal O}}
\newcommand{\RR}{{\mathbb R}}

\newcommand{\GG}{{\mathcal G}}
\newcommand{\NN}{{\mathbb N}}
\newcommand{\CC}{{\mathbb C}}

\newcommand{\CI}{{{C}^\infty}}
\newcommand{\CIb}{{{C}^\infty_b}}

\newcommand{\CIc}{{{C}^\infty_{0}}}

\newcommand{\Spec}{\mbox{\rm Spec\,}}

\renewcommand{\exp}{\mbox{\rm exp\,}}
\newcommand{\supp}{\mbox{\rm supp}}

\def\Op{{\rm Op}}

\newtheorem{defi}{Definition}
\newtheorem{theo}{Theorem}
\newtheorem{lemma}[defi]{Lemma}

\newtheorem{remark}[defi]{Remark}
\newtheorem{coro}[defi]{Corollary}
\numberwithin{equation}{section}
\numberwithin{defi}{section}

\title{Symbol calculus for semiclassical Gevrey operators}
\author{Haoren Xiong}
\email{h.xiong@bham.ac.uk}
\address{School of Mathematics, University of Birmingham, Birmingham B15 2TS, United Kingdom}

\begin{document}
	
	\begin{abstract}
		We construct a Banach algebra of Gevrey asymptotic symbols that is stable under the semiclassical quantization product and elliptic inversion (microlocal parametrix), and derive sharp quantitative estimates for the coefficients of the full asymptotic inverse. As an application, we obtain exponential estimates for adiabatic projectors in the Gevrey setting.  
	\end{abstract}
	
	\maketitle
	
	\section{Introduction}
	
	In this work, we study the symbol calculus for Gevrey semiclassical pseudodifferential operators. Semiclassical pseudodifferential operators are used to study differential equations with a small parameter $h>0$, capturing high-frequency and asymptotic behavior of solutions. Let $a\in \CIb(\RR^{2d})$ (smooth bounded functions with all derivatives bounded) be a symbol on the phase space $T^*\RR^d\cong \RR_x^d\times\RR_\xi^d$. We associate to it an semiclassical pseudodifferential operator through (standard) quantization:
	\begin{equation*}
	\begin{gathered}
		\Op_h(a) : L^2(\RR^d)\to L^2(\RR^d),\\
		 (\Op_h(a)u)(x) = \frac{1}{(2\pi h)^d} \iint_{\RR^{2d}} e^{i(x-y)\cdot\xi/h} a(x,\xi;h) u(y)\,dy\,d\xi.
	\end{gathered}
	\end{equation*}
	A cornerstone of the theory of pseudodifferential operators is that composition of such operators correspond to so-called {\textit{semiclassical product (symbol calculus)}}, denoted by $\sharp$, of the corresponding symbols, that is, 
	\begin{equation*}
		\Op_h(a)\circ\Op_h(b) = \Op_h(a\sharp b).
	\end{equation*}
	Thanks to the method of stationary phase, semiclassical products $a\sharp b$ admit an asymptotic expansion in $h$:
	\begin{equation*}
		a\sharp b \sim \sum_{\alpha\in\NN^d} \frac{h^{|\alpha|}}{i^{|\alpha|} \alpha!} \frac{\partial^\alpha a}{\partial\xi^\alpha} \frac{\partial^\alpha b}{\partial x^\alpha}.
	\end{equation*}
	We note that a central result of the semiclassical symbol calculus is the construction of an (asymptotic) elliptic parametrix: if a symbol $a$ is elliptic, i.e. its principal symbol $a_0(x,\xi)$ is nonvanishing, then one can construct a symbol $b(x,\xi;h)\sim \sum_k h^k b_k$ via asymptotic expansion such that $a\sharp b = b\sharp a =1$ modulo a negligible error. 
	
	In the classical $\CI$ setting, namely when symbols are $\CI$, the symbol calculus produces an asymptotic inverse for elliptic symbols modulo $\OO(h^\infty)$ errors. This is key to many fundamental results such as elliptic regularity, propagation of singularities, and semiclassical spectral asymptotics. In the analytic setting, developed notably by Sj\"ostrand, the symbol class satisfies factorial bounds, and the parametrix construction yields much stronger control: the remainder can be made exponentially small in $h$. This refinement plays a crucial role in problems requiring exponentially precise estimates, including quantum tunneling phenomena, resonance theory, and adiabatic theory.
	
	A natural direction is to develop symbol calculus in the Gevrey class, which provides an intermediate framework between the $\CI$ and analytic settings. Gevrey regularity allows one to go beyond the smooth theory and obtain more precise estimates on the asymptotic symbols and remainders, while retaining useful tools such as cutoffs and partitions of unity that are unavailable in the analytic setting, which makes the Gevrey regularity setting often better suited to physical models. 
	
	Let $s\geq 1$ and let $U\subset \RR^d$ be open. The Gevrey-$s$ class, denoted by $\GG^s(U)$, consists of all functions $a\in\CI(U)$ such that for every compact $K\subset U$, $\exists C,\,R>0$ s.t. 
	\begin{equation*}
		\forall\alpha\in\NN^d,\ x\in K,\quad |\partial^\alpha a(x) | \leq C R^{|\alpha|} \alpha!^s.
	\end{equation*}
	We note that when $s=1$, $\GG^1(U)$ coincides with the class of real analytic functions on $U$; while for $s>1$, the Gevrey-$s$ class is non-quasianalytic, in the sense that it contains nontrivial compactly supported functions.
	
	There are various natural examples of Gevrey functions. A classical one is the fundamental solution of the heat equation $\partial_t u-\Delta u=0$ on $\RR_t\times \RR_x^n$ given by
	\begin{equation*}
		E(t,x) = \begin{cases}
			(4\pi t)^{-n/2} \exp(-|x|^2/4t),\quad\text{if }t>0,\\
			0,\quad\text{if }t\leq 0,
		\end{cases}
	\end{equation*}
	which belongs to $\GG^2(\RR^{n+1})$. Gevrey regularity appears naturally also in wave propagation: Lebeau \cite{Lebeau} proved that diffracted waves produced by scattering off obstacles are microlocally of Gevrey class $\GG^3$, even when the obstacle is analytic and non-trapping.
	
	Let us introduce (anisotropic) Gevrey symbols on phase space. Let $\Omega\subset\RR^{2n}$ be an open set. Let $s,\,\sigma\geq 1$. We say that $a\in\CI(\Omega)$ is of $\GG_x^s \GG_\xi^\sigma$ class, denoted by $a\in\GG_x^s \GG_\xi^\sigma(\Omega)$, if for every compact $K\subset\Omega$, there exist $C,\,R>0$ such that
	\begin{equation}\label{Gevrey definition}
		\sup_{(x,\xi)\in K} |\partial_x^\alpha \partial_\xi^\beta a(x,\xi)| \leq C R^{|\alpha|+|\beta|} \alpha!^s \beta!^\sigma \quad \mbox{for any }\alpha,\,\beta \in\NN^{n}.
	\end{equation} 
	
	We then define the formal (asymptotic) Gevrey symbol classes: 
	
	\begin{defi}
		Let $s\geq 1$ and let $\{p_k\}_{k=0}^\infty$ be a sequence of functions in $\GG_x^s\GG_\xi^\sigma(\Omega)$. We say that  $\{p_k\}_{k=0}^\infty$ is a $\GG^{s,\sigma}$ sequence on $\Omega$ if for every compact $K\subset\Omega$, there exist $C,\,R>0$ such that for any $\alpha,\,\beta \in\NN^{n}$ and $k\in\NN$,
		\begin{equation}\label{eq:Gevrey symbol defn}
			\sup_{(x,\xi)\in K} |\partial_x^\alpha \partial_\xi^\beta p_k(x,\xi)| \leq C R^{|\alpha|+|\beta|+k} \alpha!^s \beta!^\sigma k!^{s+\sigma-1} .
		\end{equation} 
		Associated with a $\GG^{s,\sigma}$ sequence $\{p_k\}_{k=0}^\infty$ we can define a formal $\GG^{s,\sigma}$ symbol $p(x,\xi;h) = \sum_{k=0}^\infty h^k p_k(x,\xi)$ on $\Omega$ by the (optimal) representatives/realizations, 
		\begin{equation}
		\label{eq:optimal realization}
			p_{\Omega'}(x,\xi;h) = \sum_{k=0}^{\lfloor(Rh)^{-1/(s+\sigma-1)}\rfloor} h^k p_k(x,\xi),\quad (x,\xi)\in \Omega'\Subset \Omega,
		\end{equation}
		where $R=R_{\Omega'}$ is given as in \eqref{eq:Gevrey symbol defn} with $K=\overline{\Omega'}$. We note by \cite[Proposition 2.4]{XiongXu2024} that local representatives are equal modulo a $\GG^{s,\sigma}$-small remainder in the sense that
		\begin{equation*}
			\sup_{(x,\xi)\in\Omega_1'\cap\Omega_2'}|\partial_x^\alpha \partial_\xi^\beta (p_{\Omega_1'} - p_{\Omega_2'})| \leq C^{1+|\alpha|+|\beta|}\alpha!^s \beta!^\sigma \exp(-C^{-1}h^{-1/(s+\sigma-1)})
		\end{equation*}
		holds with a constant $C>0$ uniformly for all $\alpha,\,\beta\in\NN^n$. Here $\Omega_1'$ and $\Omega_2'$ are precompact open subsets of $\Omega$ such that $\Omega_1'\cap \Omega_2'\neq \emptyset$.
	\end{defi}
	
	The following is the main result of this work:
	
	\begin{theo}\label{thm:main}
		Let $s,\,\sigma\geq 1$, and let $\Omega\subset\RR^{2n}$ be open. Let $p=\sum_{k=0}^{\infty} h^k p_k$ be a formal $\GG^{s,\sigma}$ symbol on $\Omega$. Suppose that $p$ is elliptic in the sense that $p_0\neq 0$ everywhere on $\Omega$. Then there is a unique formal $\GG^{s,\sigma}$ symbol, $q=\sum_{k=0}^{\infty} h^k q_k$, such that $p\sharp q \sim q\sharp p \sim 1$.
	\end{theo}
	
	\begin{remark}
		The pseudodifferential calculus associated to formal Gevrey symbols traces back to \cite{BoKr}, see also \cite{Zan85}, \cite{Rodino}, where they considered $\GG_x^s\GG_\xi^1$ symbols. Our proof of Theorem \ref{thm:main} is inspired by Sj\"ostrand's version \cite[Chapter 1]{Sj82} for $\GG_x^1\GG_\xi^1$ (i.e., analytic) symbols of the argument in \cite{BoKr}. We develop a new approach that applies to the more general $\GG_x^s\GG_\xi^\sigma$ symbol class, which also provides an alternative proof of the result in \cite{BoKr}. 
	\end{remark}
	
	\begin{coro}
		Let $s,\,\sigma > 1$, and let $\Omega\subset\RR^{2n}$ be open. Let $p=\sum_{k=0}^{\infty} h^k p_k$ be an elliptic formal $\GG^{s,\sigma}$ symbol on $\Omega$. Let $q$ be its parametrix given in Theorem \ref{thm:main}. Then for every $\Omega_0\Subset\Omega_1\Subset\Omega$ and cutoff functions $\chi_j\in\GG^{s,\sigma}(\Omega_j)$ with $\supp \chi_j\subset\Omega_j$, $j=0, 1$, satisfying $\chi_1=1$ near $\supp \chi_0$, we have
		\begin{equation}\label{eq:parametrix estimate}
		\begin{gathered}
			\Op_h(\chi_0) (\Op_h(p_{\Omega_1})\Op_h(q_{\Omega_1}) - I)\Op_h(\chi_1) = \OO_{L^2 \to L^2} (e^{-C^{-1}h^{-1/(s+\sigma-1)}}), \\
			\Op_h(\chi_0) (\Op_h(q_{\Omega_1})\Op_h(p_{\Omega_1}) - I)\Op_h(\chi_1) = \OO_{L^2 \to L^2} (e^{-C^{-1}h^{-1/(s+\sigma-1)}}).
		\end{gathered}
		\end{equation} 
	Here $p_{\Omega_1}, q_{\Omega_1}$ are realizations of $p, q$ on $\Omega_1$, given as in \eqref{eq:optimal realization}.
	\end{coro}
	
	\begin{remark}
		We note that, in the case $s=\sigma>1$, Lascar \cite{Lascar1988} constructed a microlocal parametrix for elliptic Gevrey symbols with a sharper remainder $\OO(e^{-C^{-1}h^{-1/s}})$ in \eqref{eq:parametrix estimate}, which is particularly well suited to the propagation of Gevrey singularities. However, Lascar's construction is based on a finite-order symbolic inverse followed by a less explicit Neumann series argument, and does not provide a full asymptotic expansion with quantitative control of its coefficients. In contrast, our parametrix admits a complete formal $\GG^{s,\sigma}$ expansion $q=\sum_{k=0}^{\infty} h^k q_k$ with uniform Gevrey estimates (as in \eqref{eq:Gevrey symbol defn}) on the coefficients $q_k$. This property is essential in applications where the individual coefficients and their growth play an important role, such as Gevrey Bergman and Szeg\"o kernel asymptotics, Berezin--Toeplitz quantization calculus, spectral trace expansions, and quantum Birkhoff normal forms.
	\end{remark}
	
	The consideration of Gevrey (pseudo)differential operators has a long-standing tradition in the theory of PDEs, beginning with the seminal work \cite{BoKr}. Gevrey regularity problems have been studied in various contexts, including quantum theory \cite{BLR87}, \cite{Rouleux1998,Rouleux}, \cite{GaZw21}, dynamical systems \cite{MiPo10}, FBI transform \cite{Lascar1997}, \cite{GuedesFBIGevrey}, propagation of singularities \cite{Lascar1988}, \cite{Lascar1991}, \cite{Sordoni}, and pseudodifferential operators in the complex domain \cite{HiLaSjZeI, HiLaSjZeII}.
	
	\medskip
	
	In his foundational work on analytic microlocal analysis, Sj\"ostrand \cite[Chapter 1]{Sj82} introduced a family of pseudonorms for classical analytic (in our context, $\GG^{1,1}$) symbols, which give the symbolic composition a Banach algebra structure, showing that the microlocal inverse of an elliptic classical analytic symbol is again a classical analytic symbol. This functional analytic approach avoids tedious recursive type coefficient estimates that often require careful controls of factorial growth and combinatorial sums. In this work, we extend this impactful approach to the Gevrey setting, which is expected to simplify arguments of obtaining sharp asymptotic estimates in Gevrey classes, for instance, Bergman kernel coefficient estimates as in \cite{XiongXu2024}. We note that the framework developed here is not specific to the Gevrey setting and is expected to extend to broader ultradifferentiable settings, notably Denjoy--Carleman classes.

	In section \ref{sec:pseudonorms}, we introduce a family of norms for formal Gevrey symbols with similarly nice properties as the pseudonorms in \cite{Sj82}, which yield the property of a Banach algebra under the symbol calculus without calculation. We then use these norms to prove Theorem \ref{thm:main}. As an application, we apply our result to the adiabatic theory in Section \ref{sec:adiabatic}, obtaining estimates on the adiabatic projectors in the Gevrey setting.
	
	\medskip
	
	\noindent
	{\sc Acknowledgments.} The author would like to thank Michael Hitrik for valuable suggestions on the presentation and the scope of this work. The author would also like to thank Hang Xu for many inspiring discussions. Moreover, the author is grateful to Richard Lascar, Michel Rouleux, and Maciej Zworski for their valuable feedback. 
	
	\section{Formal Gevrey symbols and pseudonorms}\label{sec:pseudonorms}
	
	Let $s,\,\sigma\geq 1$, and let $\Omega\subset\RR^{2n}$ be open. For $T>0$ and $a\in\CI(\Omega)$ we introduce the following resummation:
	\begin{equation}\label{eq:Gevrey function norm}
		N_{s,\sigma}(a,T)(x,\xi) = \sum_{j=0}^\infty  \sum_{|\alpha|+|\beta|=j} {|\partial_x^\alpha \partial_\xi^\beta a(x,\xi)|}T^j /(|\alpha|!^s |\beta|!^\sigma).  
	\end{equation}
	
	\begin{lemma}\label{lemma:Gevrey function seminorms}
		Let $a\in C^\infty(\Omega)$. If $a\in\GG_x^s\GG_\xi^\sigma(\Omega)$, then for every compact $K\subset \Omega$ there exists $T_0>0$ such that $\sup_{(x,\xi)\in K} N_{s,\sigma}(a,T)(x,\xi) < \infty$, $\forall T\in [0,T_0)$. Conversely, if for every compact $K\subset \Omega$, $\sup_{(x,\xi)\in K} N_{s,\sigma}(a,T)(x,\xi) < \infty$ for some $T>0$, then $a\in\GG_x^s\GG_\xi^\sigma(\Omega)$. 
	\end{lemma}
	
	\begin{proof}
		Suppose that $a\in\GG_x^s\GG_\xi^\sigma(\Omega)$, then for every compact $K\subset \Omega$, there exist $C,\,R>0$ such that \eqref{Gevrey definition} holds. For any $T\in [0,R^{-1})$, $(x,\xi)\in K$, we have 
		\begin{equation}\label{eq:Gevrey norm bound}
			\begin{split}
				N_{s,\sigma}(a,T)(x,\xi) &\leq C \sum_{j=0}^{\infty} R^j T^j \sum_{|\alpha|+|\beta|=j} \frac{\alpha!^s}{|\alpha|!^s} \frac{\beta!^\sigma}{|\beta|!^\sigma} \\
				&\leq C \sum_{j=0}^{\infty} \binom{j+2n-1}{2n-1} (RT)^j = \frac{C}{(1-RT)^{2n}},
			\end{split}
		\end{equation}
		where we used the facts $\gamma!\leq |\gamma|!$ and $\#\{\gamma\in\NN^d : |\gamma|=j\} = \binom{j+d-1}{d-1}$.
		
		Conversely, if for every compact $K\subset \Omega$, there are constants $C,\,T>0$ such that
		\begin{equation*}
			N_{s,\sigma}(a,T)(x,\xi) = \sum_{j=0}^\infty  \sum_{|\alpha|+|\beta|=j} {|\partial_x^\alpha \partial_\xi^\beta a(x,\xi)|}T^j /(|\alpha|!^s |\beta|!^\sigma)\leq  C,\quad\forall (x,\xi)\in K.
		\end{equation*}
		Then for any $\alpha\,\beta\in\NN^n$,
		\begin{equation*}
			\sup_{(x,\xi)\in K} |\partial_x^\alpha \partial_\xi^\beta a(x,\xi)| \leq C (T^{-1})^{|\alpha|+|\beta|} |\alpha|!^s |\beta|!^\sigma \leq C (n^{s+\sigma}T^{-1})^{|\alpha|+|\beta|} \alpha!^s \beta!^s,
		\end{equation*}
		where we used the bound $|\gamma|!/\gamma!\leq n^{|\gamma|}$, $\gamma\in\NN^n$. Therefore, $a\in \GG_x^s\GG_\xi^\sigma(\Omega)$. 
	\end{proof}
	
	Let us prove next two basic properties of the resummation $N_{s,\sigma}(a,T)$, which verify that the $\GG_x^s\GG_\xi^\sigma$ class is closed under multiplication and differentiation.
	
	\begin{lemma}
		Let $a,\,b\in C^\infty(\Omega)$ and let $T>0$. We have
		\begin{enumerate}
			\item If $N_{s,\sigma}(a,T)(x,\xi),\, N_{s,\sigma}(b,T)(x,\xi)<\infty$, then
			\begin{equation}\label{eq:Gevrey multiplication}
				N_{s,\sigma}(ab,T)(x,\xi) \leq N_{s,\sigma}(a,T)(x,\xi) N_{s,\sigma}(b,T)(x,\xi).
			\end{equation}
			\item If $N_{s,\sigma}(a,T)(x,\xi)<\infty$, then for every $T_1>T$ and $\gamma\in\NN^n$,
			\begin{equation}\label{eq:Gevrey differentiation}
				N_{s,\sigma}(\partial_x^\gamma a,T)(x,\xi) \leq  (e^s (T_1^{1/s} - T^{1/s})^{-s})^{|\gamma|}|\gamma|!^s N_{s,\sigma}(a,T_1)(x,\xi).
			\end{equation}
		\end{enumerate}
	\end{lemma}
	
	\begin{proof}
		(1) By Leibniz formula, for any $\alpha,\,\beta \in\NN^n$ (all derivatives are evaluated at $(x,\xi)$ in the following),
		\begin{equation*}
			\partial_x^\alpha \partial_\xi^\beta (ab) = \sum_{\alpha_1+\alpha_2=\alpha}\sum_{\beta_1+\beta_2=\beta} \binom{\alpha}{\alpha_1} \binom{\beta}{\beta_1} \partial_x^{\alpha_1}\partial_\xi^{\beta_1} a\cdot \partial_x^{\alpha_2}\partial_\xi^{\beta_2} b.
		\end{equation*}
		Using the inequality $\binom{\mu}{\nu}\leq \binom{|\mu|}{|\nu|}$ for $\nu\leq \mu$, $\mu,\nu\in\NN^n$, and noting that $s,\sigma\geq 1$, we apply the triangle inequality and obtain
		\begin{align*}
			\sum_{\alpha,\beta\in\NN^n}\frac{|\partial_x^\alpha \partial_\xi^\beta (ab)|}{|\alpha|!^s |\beta|!^\sigma} T^{|\alpha|+|\beta|} \leq & 	\sum_{\alpha,\beta\in\NN^n} \sum_{\alpha_1+\alpha_2=\alpha}
			\sum_{\beta_1+\beta_2=\beta} \frac{|\partial_x^{\alpha_1} \partial_\xi^{\beta_1} a|}{|\alpha_1|!^s |\beta_1|!^\sigma}  \frac{|\partial_x^{\alpha_2} \partial_\xi^{\beta_2} b|}{|\alpha_2|!^s |\beta_2|!^\sigma}T^{|\alpha|+|\beta|}\\
			\leq & \sum_{\alpha_1,\beta_1\in\NN^n}\frac{|\partial_x^{\alpha_1} \partial_\xi^{\beta_1} a|}{|\alpha_1|!^s |\beta_1|!^\sigma} T^{|\alpha_1|+|\beta_1|} \sum_{\alpha_2,\beta_2\in\NN^n}\frac{|\partial_x^{\alpha_2} \partial_\xi^{\beta_2} b|}{|\alpha_2|!^s |\beta_2|!^\sigma} T^{|\alpha_2|+|\beta_2|},
		\end{align*}
		that is, \eqref{eq:Gevrey multiplication} holds, given that the two series on the right converge.
		
		(2) We first note that for any $\gamma\in\NN^n$, 
		\begin{equation*}
			N_{s,\sigma}(\partial_x^\gamma a,T)(x,\xi) = \sum_{j=0}^{\infty} \sum_{|\alpha|+|\beta|=j} \frac{|\partial_x^{\alpha+\gamma}\partial_\xi^{\beta}a|}{|\alpha|!^s|\beta|!^\sigma} T^j 
			=\sum_{j=|\gamma|}^\infty \sum_{\substack{|\alpha|+|\beta|=j \\ \alpha\geq \gamma}} \frac{|\partial_x^{\alpha}\partial_\xi^{\beta}a|}{(|\alpha|-|\gamma|)!^s|\beta|!^\sigma} T^{j-|\gamma|}.
		\end{equation*}
		Then we have for any $T_1>T$,
		\begin{equation}\label{eq:mid 1}
			\begin{split}
				N_{s,\sigma}(\partial_x^\gamma a,T)(x,\xi) \leq & T^{-|\gamma|}|\gamma|!^s \sum_{j=|\gamma|}^{\infty}\sum_{\substack{|\alpha|+|\beta|=j \\ \alpha\geq \gamma}} \frac{|\partial_x^{\alpha}\partial_\xi^{\beta}a|}{|\alpha|!^s|\beta|!^\sigma} T_1^j  \cdot {\binom{|\alpha|}{|\gamma|}}^s (T/T_1)^j \\
				\leq & T^{-|\gamma|}|\gamma|!^s N_{s,\sigma}(a,T_1)(x,\xi) \sup_{j\geq |\gamma|} {\binom{j}{|\gamma|}}^s (T/T_1)^j.
			\end{split}
		\end{equation}
		We consider therefore the sequence $A_k:= {\binom{b+k}{b}}^s \theta^{b+k}$, $k\in\NN$, for some $b\in\NN$ and $\theta\in (0,1)$. Since $A_k/A_{k-1} = (1+b/k)^s \theta$, $k\geq 1$, we have
		\begin{equation*}
			\max_{k\in\NN} A_k= A_l, \quad l=\lfloor \frac{b}{\theta^{-1/s}-1}\rfloor.
		\end{equation*}
		Noting that $\binom{b+l}{b}\leq (b+l)^b/b! \leq (b^b/b!) \theta^{-b/s} (\theta^{-1/s}-1)^{-b}$, then
		\begin{equation}\label{eq:Ak max}
			\max_{k\in\NN} A_k \leq {\binom{b+l}{b}}^s \theta^{b+l} \leq (b^b/b!)^s (\theta^{-1/s}-1)^{-bs} \theta^l \leq e^{bs} (\theta^{-1/s}-1)^{-bs} ,
		\end{equation}
		where we used Stirling's formula to get $b! > (b/e)^b$. Recalling \eqref{eq:mid 1} and applying \eqref{eq:Ak max} with $b=|\gamma|$ and $\theta=T/T_1$, we obtain \eqref{eq:Gevrey differentiation}.
	\end{proof}

	We now consider pseudodifferential calculus and parametrix construction in the class of formal  $\GG^{s,\sigma}$ symbols. To this end, let 
	\begin{equation*}
		p(x,\xi;h) = \sum_{k=0}^\infty h^k p_k(x,\xi),\quad q(x,\xi;h) = \sum_{k=0}^\infty h^k q_k(x,\xi)
	\end{equation*}  
	be formal $\GG^{s,\sigma}$ symbols on $\Omega$, the composed symbol is given by
	\begin{equation}\label{eq:composed symbol}
		r = p\sharp q = \sum_{\alpha\in\NN^n} \frac{h^{|\alpha|}}{\alpha!} i^{-|\alpha|} \frac{\partial^\alpha p}{\partial\xi^\alpha} \frac{\partial^\alpha q}{\partial x^\alpha},
	\end{equation}
	which is viewed as a formal power series in $h$ with every coefficient being a finite sum.
	
	Inspired by \cite{BoKr} and \cite{Sj82}, while we follow closely \cite{Sj82} and introduce the following differential operator of infinite order associated to a formal $\GG^{s,\sigma}$ symbol $p(x,\xi;h)$:
	\begin{equation}\label{eq:operator A}
		A(x,\xi,D_x;h) = \sum_{\alpha\in\NN^n} \frac{1}{\alpha!} \frac{\partial^\alpha p}{\partial\xi^\alpha} (x,\xi;h)(hD_x)^\alpha = \sum_{m=0}^\infty h^m A_m(x,\xi,D_x),
	\end{equation}
	where
	\begin{equation}\label{eq:operator Am}
		A_m = \sum_{k+|\alpha|=m} \frac{1}{\alpha!} \partial_\xi^\alpha p_k(x,\xi) D_x^\alpha.
	\end{equation}
	We note that one can recover the formal symbol $p$ from $A$ via the formula:
	\begin{equation*}
		p_m(x,\xi) = A_m(x,\xi,D_x)(1),\quad\forall m\in\NN.
	\end{equation*}
	
	\begin{remark}\label{remark:composition}
		Let $A$ be the associated operator of a formal classical symbol $p$, i.e. $p=\sum_k h^k p_k$, $p_k\in C^\infty(\Omega)$. If $B$ is the associated operator of another formal classical symbol $q$, then the operator associated to $r=p\sharp q$ is $C=A\circ B = \sum_{k=0}^{\infty} h^k C_k$, where $C_k = \sum_{m+l=k} A_m\circ B_l$. This is a purely algebraic result.
	\end{remark} 
	
	Let us consider the above differential operators acting on the following spaces:
	
	\begin{defi}
		Let $K\subset\Omega$ be compact and let $T>0$. Let
		\begin{equation*}
			B(K,T) = \{ a\in C^\infty(\Omega) : \sup_{(x,\xi)\in K} N_{s,\sigma}(a,T)(x,\xi) < \infty \},
		\end{equation*}
		and we equip the space $B(K,T)$ with the norm
		\begin{equation*}
			\|a\|_{K,T} := \sup_{(x,\xi)\in K} N_{s,\sigma}(a,T)(x,\xi),\quad a\in B(K,T).
		\end{equation*}
	\end{defi}
	
	\noindent
	Let us fix any compact $K\subset \Omega$. It follows from \eqref{eq:Gevrey symbol defn} that $\partial_\xi^\alpha p_k\in\GG_x^s\GG_\xi^\sigma(\Omega)$ satisfies 
	\begin{equation*}
		\begin{split}
			|\partial_{x}^\gamma \partial_\xi^\beta (\partial_\xi^\alpha p_k) (x,\xi)| &\leq C R^{k+|\alpha|+ |\beta|+|\gamma|} k!^{s+\sigma-1} {\gamma}!^s (\alpha+\beta)!^\sigma, \\
			&\leq C R^{k} (2^\sigma R)^{|\alpha|} \alpha!^\sigma k!^{s+\sigma-1}\cdot (2^\sigma R)^{|\gamma|+|\beta|} \gamma!^s \beta!^\sigma,
		\end{split}
	\end{equation*} 
	for some constants $C=C_p$, $R=R_p>0$ uniform in all $\gamma,\,\beta\in\NN^{n}$ and $(x,\xi)\in K$. Here we used the binomial bound $\frac{(\alpha+\beta)!}{\alpha!\beta!}=\binom{\alpha+\beta}{\alpha}\leq 2^{|\alpha|+|\beta|}$. Then 
	\begin{equation}\label{eq:mid proof 2}
		\sup_{(x,\xi)\in K} N_{s,\sigma}(\partial_\xi^\alpha p_k, T)(x,\xi) \leq C R^{k} (2^\sigma R)^{|\alpha|} \alpha!^\sigma k!^{s+\sigma-1} (1-2^\sigma RT)^{-2n},
	\end{equation}
	which follows from \eqref{eq:Gevrey norm bound}, provided that $T\in [0,2^{-\sigma}R^{-1})$. 
	
	\noindent
	Let $a\in B(K,T_1)$, $T_1>T$. By \eqref{eq:Gevrey differentiation}, we have
	\begin{equation}\label{eq:mid proof 3}
		N_{s,\sigma}(D_x^\alpha a,T)(x,\xi) \leq (e n)^{s|\alpha|} \alpha!^s (T_1^{1/s} - T^{1/s})^{-s|\alpha|} N_{s,\sigma}(a,T_1)(x,\xi),
	\end{equation}
	where we used $|\alpha|!\leq n^{|\alpha|}\alpha!$. Using \eqref{eq:Gevrey multiplication} together with \eqref{eq:mid proof 2}, \eqref{eq:mid proof 3}, and taking $T\in [0,2^{-s-1}R_p^{-1})$, $T_1>T$, we obtain
	\begin{equation*}
		\left\| \frac{1}{\alpha!} \partial_\xi^\alpha p_k D_x^\alpha a \right\|_{K,T} \leq  {2^{2n}C (2^\sigma e^s n^s)^{|\alpha|}} R^{k+|\alpha|}(k!\alpha!)^{s+\sigma-1} (T_1^{1/s} - T^{1/s})^{-s|\alpha|}\|a\|_{K,T_1}.
	\end{equation*}
	Denote by $\|\cdot\|_{K,T_1,T}$ the operator norm from $B(K,T_1)$ to $B(K,T)$, then 
	\begin{equation}
		\left\| \frac{1}{\alpha!} \partial_\xi^\alpha p_k D_x^\alpha \right\|_{K,T_1,T} \leq C_1^{1+k+|\alpha|} (k!\alpha!)^{s+\sigma-1} (T_1^{1/s} - T^{1/s})^{-s|\alpha|},
	\end{equation}   
	for some constant $C_1>0$ depending only on the formal $\GG^{s,\sigma}$ symbol $p$ (more precisely, the constants $C, R$ in \eqref{eq:Gevrey symbol defn} with respect to $p$ and $K$).
	
	\noindent
	The number of terms in \eqref{eq:operator Am} is $\binom{m+n}{n}\leq (m+1)^n$, thus by taking a larger $C_1$ and using $k!\alpha!\leq (k+|\alpha|)!$, we conclude that there exists $T_0=T_0(A,K)$ such that
	\begin{equation}\label{eq:operator norm Am}
		\|A_m\|_{K,T_1,T}\leq C_1^{1+m} m!^{s+\sigma-1} ((T_1^{1/s} - T^{1/s})^{s})^{-m},\quad 0\leq T <T_1\leq T_0.
	\end{equation}
	
	Conversely, if $p=\sum_{k=0}^\infty h^k p_k(x,\xi)$ is assumed to be merely a formal classical symbol on $\Omega$, that is, $p_k\in C^\infty(\Omega)$ for every $k$, and if the associated operators $A_m$, $m\in\NN$ defined by \eqref{eq:operator Am} satisfy that
	for every compact $K\subset\Omega$ the estimate \eqref{eq:operator norm Am} holds with some constants $C_1,\, T_0>0$, then $p$ must be a formal $\GG^{s,\sigma}$ symbol on $\Omega$. In fact, noting that $p_k=A_k(1)$ and $N_{s,\sigma}(1,T_1)\equiv 1$ for any $T_1>0$, we deduce that for every compact $K\subset\Omega$, the estimate \eqref{eq:operator norm Am} implies 
	\begin{equation*}
		\sup_{(x,\xi)\in K} N_{s,\sigma}(p_k,T)(x,\xi) =  \|p_k\|_{K,T} \leq C^{1+k} k!^{s+\sigma-1},
	\end{equation*}
	for some constant $C>0$. It then follows from Lemma \ref{lemma:Gevrey function seminorms} and its proof that
	\begin{equation}
		\sup_{(x,\xi)\in K} |\partial_x^\alpha \partial_\xi^\beta p_k (x,\xi)| \leq C^{1+k} k!^{s+\sigma-1} ((2n)^s T^{-1})^{|\alpha|+|\beta|} \alpha!^s \beta!^s .
	\end{equation} 
	
	Let $A$ be an infinite order differential operator given by \eqref{eq:operator A} and \eqref{eq:operator Am} with $p=\sum_{k=0}^{\infty} h^k p_k$ being a formal classical (not necessarily $\GG^{s,\sigma}$) symbol on $\Omega$. For any $K\subset\Omega$ compact and $T_0>0$, we associate to $A$ the sequence $f^K(A)=\{f_m^K(A)\}_{m=0}^\infty$ where $f_m^K(A)\in [0,+\infty]$ is the smallest number such that
	\begin{equation}\label{eq:fmA defined bound}
		\|A_m\|_{K,T_1,T} \leq f_m^K(A) (m^m)^{s+\sigma-1} ((T_1^{1/s} - T^{1/s})^{s})^{-m},\quad 0\leq T <T_1\leq T_0.
	\end{equation}
	Then \eqref{eq:operator norm Am} holds if and only if $f_m^K(A)$ is of at most exponential growth in $m$.
	
	\begin{lemma}\label{lemma:fk composition}
		Let $A=\sum_{m=0}^\infty h^m A_m$, $B=\sum_{l=0}^\infty h^l B_l$ be differential operators of infinite order associated to formal classical symbols $p,\, q$ on $\Omega$ respectively. Suppose that there is a compact $K\subset\Omega$ with constants $C_1, T_0>0$ such that \eqref{eq:operator norm Am} holds for both $A$ and $B$ (with $B_m$ in place of $A_m$). If $C = A\circ B = \sum_{k=0}^{\infty} h^k C_k$, then
		\begin{equation}\label{eq:fkC bound}
			f_k^K(C) \leq \sum_{m+l=k} f_m^K(A) f_l^K(B).
		\end{equation} 
	\end{lemma}
	
	\begin{proof}
		Let $0\leq T< T'< T_1\leq T_0$. Since \eqref{eq:operator norm Am} holds for both $A$ and $B$, we have
		\begin{equation*}
			\|A_m\|_{K,T',T} \leq f_m^K(A) (m^m)^{s+\sigma-1} ((T'^{1/s} - T^{1/s})^{s})^{-m},
		\end{equation*} 
		\begin{equation*}
			\|B_l\|_{K,T_1,T'} \leq f_l^K(B) (l^l)^{s+\sigma-1} ((T_1^{1/s} - T'^{1/s})^{s})^{-l},
		\end{equation*}
		Let us choose $T'$ such that 
		\begin{equation*}
			T_1^{1/s} - T'^{1/s} = (T_1^{1/s} - T^{1/s})\frac{l}{m+l},\quad T'^{1/s} - T^{1/s} = (T_1^{1/s} - T^{1/s})\frac{m}{m+l}.
		\end{equation*}
		Then we obtain
		\begin{equation*}
			\begin{split}
				\|A_m\circ B_l\|_{K,T_1,T} &\leq f_m^K(A) f_l^K(B) (m^m l^l)^{\sigma-1} (m+l)^{(m+l)s} ((T_1^{1/s} - T^{1/s})^s)^{-(m+l)} \\
				&\leq f_m^K(A) f_l^K(B) ((m+l)^{m+l})^{s+\sigma-1} ((T_1^{1/s} - T^{1/s})^s)^{-(m+l)}
			\end{split}
		\end{equation*}
		Since $C_k = \sum_{m+l=k} A_m\circ B_l$, we conclude
		\begin{equation*}
			\|C_k\|_{K,T_1,T} \leq \sum_{m+l=k} f_m^K(A) f_l^K(B) (k^k)^{s+\sigma-1} ((T_1^{1/s} - T^{1/s})^s)^{-k},
		\end{equation*}
		which implies \eqref{eq:fkC bound} by definition.
	\end{proof}
	
	For $\rho>0$ and $K\subset\Omega$ compact, let us follow \cite{Sj82} and set 
	\begin{equation*}
		\|A\|_{K,\rho} = \sum_{m=0}^{\infty} \rho^m f_m^K(A).
	\end{equation*}
	Then \eqref{eq:operator norm Am} holds for some $C_1>0$ if and only if $\|A\|_{K,\rho}<\infty$ for some $\rho>0$.
	
	\begin{lemma}\label{lemma:Bananch algebra}
		Let $T_0>0$ and $K\subset\Omega$ be compact. Let $\rho>0$. If $\|A\|_{K,\rho}<\infty$ and $\|B\|_{K,\rho}<\infty$, and if $C=A\circ B$, then $\|C\|_{K,\rho}<\infty$ and satisfy
		\begin{equation*}
			\|C\|_{K,\rho} \leq \|A\|_{K,\rho} \|B\|_{K,\rho}.
		\end{equation*}
	\end{lemma}
	
	\begin{proof}
		By Lemma \ref{lemma:fk composition} we have	
		\begin{equation*}
			\|C\|_{K,\rho} = \sum_{k=0}^\infty \rho^k f_k^K(C) \leq \sum_{k=0}^\infty \sum_{m+l=k} \rho^{m+l} f_m^K(A) f_l^K(B) = \|A\|_{K,\rho} \|B\|_{K,\rho}.
		\end{equation*}
	\end{proof}
	
	Let $\|p\|_{K,\rho} = \|A\|_{K,\rho}$, where $A$ is the associated operator of $p$. If $p$ is a formal $\GG^{s,\sigma}$ symbol on $\Omega$, then for every $K\subset\Omega$ compact \eqref{eq:operator norm Am} holds with some $C_1, T_0>0$ and thus $\|p\|_{K,\rho}<\infty$ for some $\rho>0$. Conversely, if for every $K\subset\Omega$ compact, there is a $\rho>0$ such that $\|p\|_{K,\rho}<\infty$ (which also indicates that there is $T_0>0$ such that \eqref{eq:fmA defined bound} holds), then $p$ is a formal $\GG^{s,\sigma}$ symbol on $\Omega$. Therefore, in view of Lemma \ref{lemma:Bananch algebra} and Remark \ref{remark:composition}, we obtain	
	\begin{coro}\label{coro:symbol Banach algebra}
		If $p$ and $q$ are formal $\GG^{s,\sigma}$ symbols on $\Omega$, then so is $p\sharp q$, and
		\begin{equation*}
			\|p\sharp q\|_{K,\rho} \leq \|p\|_{K,\rho} \|q\|_{K,\rho} .
		\end{equation*}
	\end{coro}
	
	\medskip
	
	We are now ready to prove Theorem \ref{thm:main}. From the standard parametrix construction, we know that $q=\sum_{k=0}^{\infty} h^k q_k$ exists as a formal classical symbol. Moreover, let $q_0=\frac{1}{p_0}$, 
	\[
	p\sharp q_0 = 1 - r,\quad r=\sum_{k=1}^{\infty} h^k r_k(x,\xi).
	\]
	Then $q=q_0\sharp (1+r+r\sharp r + \cdots)$ (a finite sum in each power of $h$) in the asymptotic sense. It remains to show that $q$ is a formal $\GG^{s,\sigma}$ symbol on $\Omega$. From Corollary \ref{coro:symbol Banach algebra}, we see that $r$ is a formal $\GG^{s,\sigma}$ symbol. Let us fix any $K\subset\Omega$ compact. There is $T_0>0$ such that \eqref{eq:fmA defined bound} holds for the operator $R=\sum_{k=0}^\infty R_k$ associated to $r$. Since $r$ is a symbol of order $-1$, i.e. $r_0=0$, we have $R_0=0$ and so $f_0^K(R)=0$, and
	\[
	\|r\|_{K,\rho} = \sum_{k=0}^\infty \rho^k f_k^K(R) = \rho \sum_{k=0}^\infty \rho^k f_{k+1}^K(R) = \OO(\rho),\quad 0<\rho\leq\rho_0,
	\]   
	for some $\rho_0>0$. We set $\rho>0$ sufficiently small so that $\|r\|_{K,\rho}<1/2$. Then
	\[
	\|\underbrace{r\sharp\cdots\sharp r}_{m}\|_{K,\rho}\leq 2^{-m}.
	\]
	We obtain therefore $\|1+r+r\sharp r + \cdots\|_{K,\rho}\leq 2$. As a result, $(1+r+r\sharp r + \cdots)$ is a formal $\GG^{s,\sigma}$ symbol, since $q_0=1/p_0\in\GG_x^s\GG_\xi^\sigma(\Omega)$ is obviously a formal $\GG^{s,\sigma}$ symbol, we conclude that $q=q_0\sharp (1+r+r\sharp r + \cdots)$ is a formal $\GG^{s,\sigma}$ symbol. This completes the proof of Theorem \ref{thm:main}.  
	
	\section{Adiabatic projectors}\label{sec:adiabatic}
	
	In this section we apply Theorem \ref{thm:main} to adiabatic theory and recover exponential estimates for adiabatic expansions. We also obtain exponential estimates for a class of frequency-filtered adiabatic evolution problems. Let us start with a brief review of adiabatic evolution.
	
	Let $\HH$ be a separable complex Hilbert space.
	Let $P(t) : \HH\to\HH$, $t\in [0,1]$, be a family of closed densely defined unbounded operators with domain $\DD=\DD(t)\subset\HH$ independent of $t$ such that
	\begin{enumerate}
		\item The map $t\mapsto P(t)\in \LL(\DD,\HH)$ is of $\CI$ class in the natural sense: the successive derivatives
		are uniformly bounded $\DD\to\HH$ in $t\in[0,1]$ and again differentiable; 
		\item There exist two continuous curves $\gamma_j: [0,1]\to\RR$, $j=1,2$, with $\gamma_1(t)<\gamma_2(t)$, $\gamma_j(t)\notin \Spec(P(t))$, $j=1,2$, $t\in [0,1]$.
	\end{enumerate}
	Let $\Sigma_0(t) = [\gamma_1(t),\gamma_2(t)]\cap \Spec(P(t))$. Then the spectral projector $\Pi_0(t)={\mathbf 1}_{\Sigma_0(t)}(P(t))$ is of class $\CI([0,1];\LL(\HH,\DD))$, and satisfies
	\begin{equation}\label{eq:Pi_0}
		\Pi_0(t) = \frac{1}{2\pi i} \int_{\Gamma_0(t)} (z-P(t))^{-1} \,dz,
	\end{equation}
	where $\Gamma_0$ is the circle symmetric with respect to the real axis and passing through the points $\gamma_j(t)$, $j=1,2$.
	
	We are concerned about adiabatic evolution equations of the form
	\begin{equation}
		(hD_t + P(t))u(t) = 0,\quad 0\leq t\leq 1,
	\end{equation}
	for small $h>0$, and their solutions close to the image of $\Pi_0(t)$. To this end, we view $hD_t + P(t)$ as a vectored valued (pseudo)differential operator with symbol $\tau+P(t)$. Let us recall from the standard parametrix construction for elliptic pseudodifferential operators that there is a unique formal classical symbol taking values in $\LL(\HH,\DD)$:
	\begin{equation*}
		S=S(t,\tau-z;h) \sim S_0(t,\tau-z) + h S_1(t,\tau-z) + \cdots + h^k S_k(t,\tau-z) + \cdots
	\end{equation*}
	for $t\in[0,1]$, $z-\tau\in\CC\setminus\Spec(P(t))$, such that
	\begin{equation}\label{eq:S as inverse symbol}
		(z-\tau-P(t)) \sharp S \sim \text{id}_{\HH},\quad S\sharp (z-\tau-P(t)) \sim \text{id}_{\DD},
	\end{equation}
	and we note that $S_0(t,\tau-z) = (z-\tau-P(t))^{-1}$.
	
	For $\tau\in\text{neigh}\,(0;\CC)$, we set the formal classical symbol 
	\begin{equation}
		\Pi(t,\tau;h) = \frac{1}{2\pi i} \int_{\Gamma_0(t)} S(t,\tau-z;h) \,dz,
	\end{equation}
	with $\Gamma_0(t)$ as in \eqref{eq:Pi_0}. It is shown in \cite{Sj93} that $\Pi(t,\tau;h)$ is independent of $\tau$ due to the holomorphicity of $S(t,\tau-z;h)$ in $z$. Therefore,
	\begin{equation*}
		\Pi(t,hD_t;h)u = \Pi(t;h)u(t),
	\end{equation*}
	where $\Pi(t;h) \sim \sum_{j=0}^{\infty} h^j \Pi_j(t)$ is given by
	\begin{equation}\label{eq:Pi_j}
		\Pi_j(t) = \frac{1}{2\pi i} \int_{\Gamma_0(t)} S_j(t,-z) \,dz\in\CI([0,1];\LL(\HH,\DD)),\quad j=0,1,\ldots.
	\end{equation}
	Furthermore, it is proved in \cite{Sj93} that 
	\begin{equation}\label{eq:projector prop}
		\Pi(t;h)\circ\Pi(t;h) \sim \Pi(t;h),
	\end{equation}
	\begin{equation}\label{eq:projector commute}
		[hD_t + P(t), \Pi(t;h)] \sim 0.
	\end{equation}
	
	We remark that one can deduce the growth of $\|\Pi_j(t)\|_{\LL(\HH,\DD)}$ in $j$ from the parametrix of $z-(hD_t+P(t))$, more precisely, from properties of the formal symbol $S(t,\tau-z;h)$ in \eqref{eq:S as inverse symbol}. For instance, using results of \cite{BoKr} in the form given in \cite[Chapter 1]{Sj82}, Sj\"ostrand considered in \cite{Sj93} the case where $t\mapsto P(t)$ extends to a holomorphic (operator-valued) function in a complex neighborhood of $[0,1]$ and obtained that $S(t,\tau-z;h)$ is then a formal classical analytic operator-valued symbol. It follows that operators $\Pi_j(t)$ are holomorphic in a complex neighborhood $V$ of $[0,1]$ satisfying that for some $C>0$,
	\begin{equation}
		\sup_{t\in V} \|\Pi_j(t)\|_{\LL(\HH,\DD)}  \leq C^{j+1} j!.
	\end{equation}
	This estimate was also obtained in \cite{Ne93} and \cite{JoPf}. For the discussion of the optimal constant $C$, we refer to \cite{Ma94}.
	
	Using Theorem \ref{thm:main}, we extend the above estimates to the Gevrey cases:
	
	\noindent
	{\bf Example 1.} Let us assume that $t\mapsto P(t)$ is of Gevrey-$s$ ($s>1$) class, that is
	\begin{equation*}
		\exists C>0,\quad\left\|P^{(k)}(t)\right\|_{\LL(\DD,\HH)} \leq C^{k+1} k!^s,\quad t\in [0,1].
	\end{equation*}
	Then in our notation, we have $z-\tau-P(t)\in\GG_t^s \GG_\tau^1$ is elliptic provided that $t\in [0,1]$, $z-\tau\in\CC\setminus\Spec(P(t))$. By Theorem \ref{thm:main}, we deduce that $S=\sum_{j=0}^\infty h^j S_j(t,\tau-z)$ is a formal $\GG^{s,1}$ symbol, and therefore, for some $C>0$ we have
	\begin{equation}\label{eq:adiabatic est Gevrey}
		\sup_{t\in [0,1]} \|\Pi_j(t)\|_{\LL(\HH,\DD)}  \leq C^{j+1} j!^s.
	\end{equation}
	This recovers results of Nenciu \cite{Ne93}. We note that our arguments to obtain \eqref{eq:adiabatic est Gevrey} was predicted by Sj\"ostrand in \cite{Sj93}, where he suggested using Bontet de Monvel--Kr\'ee's theory \cite{BoKr} for Gevrey parametrix in the form similar as in \cite[Chapter 1]{Sj82}. 
	
	\noindent
	{\bf Example 2.} Let $\sigma>1$. Let $a\in \GG^\sigma(
	\RR)\cap\CIc(\RR)$ be such that $a=0$ on $(-\infty,0]\cup [1,\infty)$ and $a>0$ in $(0,1)$. Let $P(t)\in\GG_b^s([0,1];\LL(\DD,\HH))$ be the same as in Example 1.
	We consider the frequency-filtered adiabatic evolution equation:
	\begin{equation}\label{eq:freq filtered adiabatic}
		(a(hD_t)+P(t))u(t)=0,\quad 0\leq t\leq 1.
	\end{equation}
	For $t\in [0,1]$, $z-a(\tau)\in\CC\setminus\Spec(P(t))$, the operator-valued symbol $z-a(\tau)-P(t)$, in our notation, is an elliptic $\GG_t^s\GG_\tau^\sigma$ symbol. We obtain therefore by Theorem \ref{thm:main} that its parametrix $S=\sum_{j=0}^\infty h^j S_j$ is a formal $\GG^{s,\sigma}$ symbol, which implies, for some $C>0$,
	\begin{equation}
		\sup_{t\in [0,1]} \|\Pi_j(t)\|_{\LL(\HH,\DD)}  \leq C^{j+1} j!^{s+\sigma-1},
	\end{equation}
	As a consequence, our asymptotic projection $\Pi(t;h)\sim \sum_{j=0}^\infty h^j \Pi_j(t)$ satisfies \eqref{eq:projector prop} and \eqref{eq:projector commute} with a quantitative remainder estimate if we truncate the asymptotic optimally, namely an error of size $\OO (e^{-C^{-1}h^{-1/(s+\sigma-1)}})$. We can therefore obtain almost invariant spectral subspaces with fractional exponential errors for the model 
	\eqref{eq:freq filtered adiabatic} as in \cite{Ne93}, which can be viewed as a time-dispersive generalisation of the standard adiabatic theory.

\end{document}